\documentclass[12pt, reqno]{amsart}
\allowdisplaybreaks[1]
\usepackage{amsmath}
\usepackage{amssymb}
\usepackage{amsfonts}
\usepackage{verbatim}
\usepackage[usenames]{color}
\usepackage{hyperref}
\makeindex

 \newtheorem{theorem}{Theorem}[section]
 \newtheorem{corollary}[theorem]{Corollary}
 
 \newtheorem{proposition}[theorem]{Proposition}

\theoremstyle{definition}

\theoremstyle{remark}

\newtheorem{fact*}{Fact}

\newcommand\dd{\mathrm d}

\newcommand{\til}{\raise.17ex\hbox{$\scriptstyle\mathtt{\sim}$}}

\newcommand\beq{\begin{equation}}

\newcommand\eeq{\end{equation}}

\newcommand{\bbm}{\left[ \begin{smallmatrix}}
\newcommand{\ebm}{\end{smallmatrix} \right]}
\newcommand{\bpm}{\left( \begin{smallmatrix}}
\newcommand{\epm}{\end{smallmatrix} \right)}
\numberwithin{equation}{section}

\newlength{\Mheight}
\newlength{\cwidth}

\newcommand{\dfn}[1]{{\bf #1}\index{#1}}

\title{Germination phenomena}
\author{
J. E. Pascoe 
}
\date{\today}
\thanks{Partially supported by National Science Foundation DMS Analysis Grant 1953963}

\subjclass[2020]{32A10, 32C35, 12H25}

\begin{document}

\begin{abstract}
Many theorems in complex analysis propagate analyticity, such as the Forelli theorem, edge-of-the-wedge theorem and so on. We give a germination theorem which allows for general analytic propagation in complete normed fields. In turn, we develop general analogs of the Forelli theorem, edge-of-the-wedge theorem, and the royal road theorem, and gain insight into the geometry of the zero sets of hyperbolic polynomials.
\end{abstract}
\maketitle


\section{Introduction}
Recall the following theorem.
\begin{theorem}[F. Forelli \cite{forelli}]
Let $B^d$ denote the Euclidean unit ball in $\mathbb{C}^d$
Let $f:B^d \rightarrow \mathbb{C}$ such that $f$ is smooth at $0$ and
given any vector $x$ the function of one variable $f(zx)$ is analytic.
Then, $f$ is analytic on $B^d.$
\end{theorem}
Forelli's theorem has been generalized and extended in various ways \cite{ktk2013,ktk2016, ChoLocal,krantz2018}. Namely, there are independent generalizations \cite{pascoeblmswedge,pascoecmb, ChoLocal} which functions smooth at $0$ which are analytic in a relatively small set of directions having some kind of analytic continuation to a neighborhood of $0.$

Recall another theorem.
\begin{theorem}[The edge-of-the-wedge theorem, Bogoliubov]
Let $\Pi$ denote the upper half plane in $\mathbb{C}.$
Let $D \subseteq \mathbb{R}^d$ be open.
Let $f:\Pi^d\cup D \cup (-\Pi)^d \rightarrow \mathbb{C}$ be continuous such that $f$ is analytic
on $\Pi^d\cup (-\Pi)^d.$
Then, $f$ analytically continues to a neighborhood of $D$ in $\mathbb{C}^d.$
\end{theorem}
For more information about the edge-of-the-wedge theorem and its history, see Rudin \cite{rudeow}. Occasionally, the geometry of $D$ can cause continuation within $\mathbb{R}^d$ itself, which is the content of the wedge-of-the-edge theorems, see \cite{pascoeblmswedge,pascoecmb}. In turn, the wedge-of-the-edge theorems were used to prove vast generalizations of automatic analyticity theorems (often operator theoretic) via the ``royal road" technique \cite{royal, regal} such as the L\"owner theorem on matrix monotone functions \cite{lo34, pastdfree, pascoeopsys, amyloew, pal1} and the Kraus theorem on matrix convex functions \cite{kraus36, pal2}. Essentially, the royal road theorem says that a class of locally integrable functions which is convex and closed in the topology of pointwise convergence whose members are analytic on positively oriented slices is a controlled way must be a class of analytic functions. We give a general analogue of the edge-of-the-wedge theorem in Theorem \ref{eow2} and the ``royal road" in Theorem \ref{nobleclass}.

Philosophically, one expects theorems in complex variables to really depend on the fact that one is dealing with complex numbers and their special geometric and algebraic properties, as opposed to some kind of general analysis or set theoretic ones. We prove a germination theorem which shows that a power series over a complete normed field (namely including the $p$-adics and completions of function fields, and relevant to $p$-adic analysis \cite{robert2013course} and differential equations \cite{padicdiffeq}) which is directionally convergent does indeed define an analytic function on an open set. In turn, we obtain some automatic analyticity results for certain sheaves of functions. Additionally, we will see a Forelli theorem for polydisks over an algebraically closed normed field, and an edge-of the-wedge theorem for general normed fields. Hence, one may be led to the conclusion that the aforementioned theorems lack a true ``complex aesthetic."

\section{Some preliminaries}
\subsection{Normed fields}
	Let $\mathbb{F}$ be a field. A \dfn{norm} on $\mathbb{F}$ is a function
	$|\cdot|:\mathbb{F}\rightarrow [0,\infty)$ such that
	\begin{enumerate}
		\item $|x|=0$ if and only if $x=0,$
		\item $|xy|=|x||y|,$
		\item $|x+y|\leq |x|+|y|.$
	\end{enumerate}
	We say a normed field $\mathbb{F}$ is \dfn{complete} if it is complete with respect to the
	metric $d(x,y) = |x-y|.$ We say a normed field is \dfn{nonarchimedian} if all triangles are isosceles, otherwise we call it archimedian.
	
	Ostrowski's theorem classifies the norms on $\mathbb{Q}$ \cite{robert2013course}.
	We call a normed field \dfn{local} if it is locally compact. Local fields, being themselves
	locally compact groups,
	possess a unique translation invariant Borel measure
	such that the unit ball has measure $1,$ the so-called \dfn{Haar measure.}
	
\subsection{The Cantor space}
	Define a \dfn{bitstring} to be a finite word composed from the alphabet $\{0,1\}.$
	Define the \dfn{bitstring binary tree}
	on bitstrings such that a bitstring $\alpha$ is the parent of $\alpha 0$
	and  $\alpha 1.$ The \dfn{Cantor space} $2^{\aleph_0}$ is the set of branches of the of the bitsting binary tree equipped with the metric that the distance from $a$ to $b$ is equal to $2^{-n}$
where $n$ is depth to which the branches $a$ and $b$ agree. There is the natural measure (in fact a Haar measure) on
$2^{\aleph_0}$ given by saying the set of branches extending a particular bitstring $\alpha$ of length $n$ has measure $2^{-n}.$

\section{The spread sets and interpolation}
Let $\mathbb{F}$ be a complete normed field.
For a polynomial $h$ define $\|h\|$ to be the sum of the norms of its coefficients, and define
$\|h\|_X =\sup_{X} |h|.$
Let $\mathbb{B}$ be the set of elements of $\mathbb{F}$ with norm less than $1.$
Say an embedding $\iota: 2^{\aleph_0}\rightarrow \mathbb{F}$ is \dfn{spread}
if
	$d(\iota(x),\iota(y))\geq Cd(x,y)^\gamma$
for some $C, \gamma >0.$
Say $X\subseteq \mathbb{F}^d$ is \dfn{affine-spread} if there are $\iota_1,\ldots,\iota_d$ spread embeddings such that $\iota=\iota_1 \times \ldots\times \iota_d$ satisfies $\iota^{-1}(X)$ has positive outer measure with respect to the natural measure on  $(2^{\aleph_0})^d.$
We say $X \subseteq \mathbb{F}^d$ is \dfn{spread} if the set
$$\tilde{X}=\{(x_2/x_1,\ldots, x_d/x_1)|(x_1,\ldots, x_d)\in X, x_1\neq 0\}$$
is affine-spread.

We note that for any local field $\mathbb{F},$ any set $X$ with positive measure is affine-spread.
\begin{proposition}
Let $\mathbb{F}$ be a local field. Any positive measure subset $X$ of $\mathbb{F}^d$ is affine-spread.
\end{proposition}
\begin{proof}
Let $p\in \mathbb{F}$ such that $|p|<1/2.$
Consider the set $$S = \{\sum^{\infty}_{i=1} a_ip^i | a_i \in \{0,1\}\}.$$
Note that $S^d$ is affine-spread.
Induce a measure $\nu$ on $\mathbb{F}^d$ by taking the natural measure on $S^d$ by viewing 
$S$ as a Cantor space.
Consider   $a(t)=\int 1_X(x-t) \dd\nu(x).$
Now $$\int a(t) dt = \int \int 1_X(x-t) \dd\nu(x) \dd t = \int \int 1_X(x-t) \dd t \dd\nu(x) = |X|.$$
Thus, there must be some t such that $a(t)\neq 0$ and so $\nu(S+t \cap X)>0$.
\end{proof}

Let $X\subseteq \mathbb{F}^d.$ We say $X$ has \dfn{exponentiallly conditioned interpolation}
if there are $A, B>0$ such that
for any polynomial $h$
$$\|h\| \leq AB^{\deg h}\|h\|_X.$$

\begin{proposition}[Spread interpolation theorem]\label{interbound}
	Let $X$ be affine-spread. Then, $X$ has exponentially conditioned interpolation.
	
\end{proposition}
\begin{proof}
	The proof will go be induction on the dimension. If $d=0,$ the claim is trivial.
	Let $n= \deg h.$
	Let $\mu$ be the measure of  $\iota^{-1}(X).$
	Choose $m$ such that $2^m \mu >n.$ 
	Choose $\rho_0, \ldots, \rho_n$ such that each of the slices $\{\rho_i\}\times \mathbb{F}^{d-1} \cap X = \{\rho_i\} \times X_i$  satisfies $X_i$ is spread and $\iota_1^{-1}(\rho_i)$ are in distinct subtrees of the Cantor space at level $m.$
	By Lagrange interpolation,
		$$h=\sum h_i \prod_{j\neq i} \frac{z-\rho_j}{\rho_i-\rho_j}.$$
Now, $h_i$ by induction and the numerators by the fact that $X$ must be bounded are obviously exponentially bounded. So, it suffices to bound $\prod_{j\neq i}(\rho_i-\rho_j)$ below.
Now $\rho_i-\rho_j \geq Cd(\iota_1^{-1}(\rho_i),\iota_1^{-1}(\rho_j))^{\gamma}.$
Thus it is sufficient to bound $$-\frac{1}{n}\sum_{j\neq i} \log_2 d(\iota_1^{-1}(\rho_i),\iota_1^{-1}(\rho_j)).$$
Since these are in distinct branches of the Cantor space, note 
the worst case estimate for the average would be
	$$\frac{(m-1) + (m-2)2 + (m-3)2^2 + \ldots + (m-\lfloor{\log_2 n}\rfloor)2^{\lfloor{\log_2 n}\rfloor}}{2^{\lfloor{\log_2 n}\rfloor}}$$
which is bounded by $2^{2+(m-\lfloor\log_2 n\rfloor)}\leq 2^{3-\log_2 \mu}.$
\end{proof}

\section{The germination theorem}

We now prove our germination theorem.
\begin{theorem}[Germination theorem]
	Let $X$ be spread. Let $f$ be a formal power series. Suppose that for each $x\in X$ we have that $f(zx)$ has positive radius of convergence as a function of $z \in \mathbb{F}.$ Then, the formal power series $f$ has positive radius of convergence.
\end{theorem}
\begin{proof}
	Write $f(x) = \sum h_n(x)$ where $h_n$ is a homogeneous polynomial of degree $n.$
	Let $A_{C,K}$ denote the set of $x\in X$ such that the power series coefficients $h_n(x)$ are bounded by $CK^n.$
	Note since their union is everything, and they are nested and relatively closed, some $\iota^{-1}(A_{C,K})$ has positive outer measure and is thus spread. Applying the spread interpolation theorem to each of the $h_n$, we are done.
\end{proof}

The following is a quantitative variant of the germination theorem which gives more information about the radius of convergence when we know the radii of convergence \emph{a priori} in a spread set of directions.
\begin{theorem}[Quantitative germination theorem]
	Let $X$ be spread. Let $f$ be a formal power series. There is a $C>0$ such that if we for each $x\in X$ we have that $f(zx)$ has radius of convergence $1$ as a function of $z \in \mathbb{F},$ then the formal power series $f$ has radius of convergence at least $C$.
\end{theorem}
\begin{proof}
	The proof will go by contradiction.
	Let $f_n$ be a sequence of functions satisfying our hypotheses such that their radii of convergence go to zero. Let $\varepsilon >0.$ Write the homogeneous expansion $f_n = \sum h_{m,n}.$ 
	Inductively define $A_{K,n}$ to be the set of $x\in \overline{X}$ such that $|h_{m,n}(x)| \leq K(1+\varepsilon)^m.$ Note for each $n$ the union of $A_{K,n}$ is all of $X$
	Note $X\setminus A_{K,n}$ is relatively open in $X$ and therefore the measure of $\iota^{-1}(X\setminus A_{K,n})$ goes to $0$ as $n$ goes to infinity. For each $n$, choose $K_n$ such that the measure of $\iota^{-1}(X\setminus A_{K_n,n})<1/n!.$
	So, there is an $N$ such that  the preimage of
	$\tilde{X}=\bigcap_{n\geq N} (A_{K_n,n})$ has positive outer measure and is thus spread. Applying the spread interpolation theorem to each $h_{m,n}$ gives a uniform radius of convergence, which is a contradiction.
\end{proof}

\subsection{The geometry of hyperbolic polynomials}
Let $\mathbb{F}$ be a normed field. Let $X\subseteq \mathbb{F}^d$ be spread.
We say a polynomial $p$ is \dfn{$X$-hyperbolic} if for every $u \in \mathbb{F}^d$ and $x \in X$ the polynomial
$p(u+zx)$ is $\mathbb{F}$-rooted.
We see the following corollary of the quantitative germination theorem, which says that if the zeros are far away along a spread set of slices, then the zeros are far away.
\begin{corollary}
	Let $\mathbb{F}$ be a complete normed field. Let $X$ be spread.
	Suppose $p$ is an $\mathbb{F}$-hyperbolic polynomial on $\mathbb{F}^d$.
	Let $Z$ be the zero set of $p.$
	There is a $C>0$ depending only on $X$ such that for every $u \in \mathbb{F}^d,$
	$$d(u, Z)\geq C\inf_{x\in X} \sup_r \{r | p(u+zx)\neq 0 \textrm{ for all } |z|\leq r  \}.$$
\end{corollary}

\section{The perfect interpolation property and Forelli's theorem on polydisks}

We say $\mathbb{F}$ has the \dfn{perfect interpolation property} if given a nested sequence of closed sets $X_i$ such that $$\bigcup X_i = \mathbb{B}^d$$ there are $B_i \rightarrow 1$ and a companion sequence $A_i>0$ such that for any polynomial $h$ we have that
$$\|h\| \leq A_iB_i^{\deg h}\|h\|_{X_i}.$$

\begin{theorem}
	Let $\mathbb{F}$ be an algebraically closed complete normed field.
	Then, $\mathbb{F}$ has the perfect interpolation property.
\end{theorem}
\begin{proof}
Suppose $\mathbb{F}$ is archimedian. 
The proof will go by induction on the number of variables. For the $d=0$ case, the claim is obvious.
Pick some $r<1$
Consider the set
$$S_{r}=\{|z|=r\}.$$
Choose $N$ such that $X_N \cap S_r$ has relative measure greater that $1-\varepsilon.$
Let $\rho_0 \ldots \rho_n \in X_n$ such that they are some rotated subset of the 
$\lceil \frac{n}{1-\varepsilon} \rceil$-roots of unity times $r$.
Let $t =n/\lceil \frac{n}{1-\varepsilon} \rceil.$
By Lagrange interpolation,
		$$h=\sum h_i \prod_{j\neq i} \frac{z-\rho_j}{\rho_i-\rho_j}.$$
where the $h_i$ are polynomials in one less variable.
Note,
$$\log \prod|\rho_i-\rho_j| \geq \frac{n}{2t}\int^{t}_{-t} \log|1-e^{i\pi \theta}|d\theta + n \log r.$$
Moreover,
$$\log \prod |e^{i\theta} -\rho_j| \leq \frac{n}{2t}\int^{t}_{-t} \log|1-e^{i\pi(1+\theta)}|d\theta.$$
As $\varepsilon \rightarrow 0, r\rightarrow 1,$ we see that these integrals go to $0,$ so we are done.

On the other hand, suppose $\mathbb{F}$ is nonarchimedian.
The proof will go by induction on the number of variables. For the $d=0$ case, the claim is obvious.
Let $p\in \mathbb{F}$ such that $|p|<1.$
Consider the set
	$$S_{p,k}=\{\sum^{\infty}_{n=1} a_np^{n/k}| a_n \in \{0,1\}\}.$$
Note $S_{p,k}$ is affine-spread with constants
$\gamma = \log_2 |p|^{-1/k}, C=1.$
Pick $N$ such that $X_N\cap S_{p,k}$ has measure greater than $1/2.$
Pick $\rho_0,\ldots, \rho_n \in X_N$ in distinct branches of the $\lfloor \log_2 n+2\rfloor$ level of the Cantor space.
By Lagrange interpolation,
		$$h=\sum h_i \prod_{j\neq i} \frac{z-\rho_j}{\rho_i-\rho_j}.$$
where the $h_i$ are polynomials in one less variable.
Now 
	$|\prod_{j\neq i} \frac{z-\rho_j}{\rho_i-\rho_j}|$ for $z \in \mathbb{B}$
is bounded by $|\prod_{i\neq j} \rho_i-\rho_j|^{-1}\leq  |p|^{-Gn/k}$
for some $G>0$ as in the proof of  spread interpolation theorem. Note as we increase $k$
we can make the growth rate arbitrarily close to $1,$ so we are done.
\end{proof}
We note that neither $\mathbb{R}$ nor $\mathbb{Z}_p$ have the perfect interpolation property.
For example, the family of polynomials $((z-3z^3)/2)^n$ are bounded by $1$ on the ball
in $\mathbb{R}$ and $\mathbb{Z}_2,$ but the coefficient of $z^{3n}$ grows without bound.

\begin{theorem}
	Suppose $\mathbb{F}$ has the perfect interpolation property.
	Suppose $f$ is a formal power series and
	$f(xz)$ has radius of convergence at least $1$ for $x\in \mathbb{B}^d.$
	Then, $f$ defines an analytic function on $\mathbb{B}^d$.
\end{theorem}

\section{Automatic analyticity}
We now discuss more abstract, flexible analogues of the royal road \cite{royal,regal} which in the complex case has been used to prove automatic analyticity results for multivariate matrix monotone and matrix convex functions, including a short proof of the commutative two variable L\"owner theorem due to Agler, McCarthy and Young \cite{amyloew}.

Call a sheaf of smooth functions $\mathcal{F}$ on $\mathbb{F}^d$ \dfn{pre-gentric}
if whenever $U$ is an open set, $u\in U,$ $f\in \mathcal{F}(U)$ and the formal power series for $f$ at $u$ is convergent on $U,$ then $f$ is equal to its formal power series expansion. 
Call a pre-gentric sheaf of functions $\mathcal{F}$ on $\mathbb{F}^d$ \dfn{gentric} if
there is a spread $X$ such that
for every open set $U$ and every $u\in U$, for any $f \in \mathcal{F}(U)$ the function $f(u+xz)$ is analytic in the single variable $z$ for every $x \in X$ near $z=0.$
We see the following as a result of the germination theorem.
\begin{corollary}
Every function in a gentric sheaf is analytic.
\end{corollary}

Call a sheaf of locally bounded functions $\mathcal{F}$ on $\mathbb{F}^d$ \dfn{noble} if
\begin{enumerate}
\item there is a spread $X$ such that
for every open set $U$ and every $u\in U$, for any $f \in \mathcal{F}(U)$ the function
$f(u+xz)$ is analytic in the single variable $z$ for every $x \in X,$
\item for each $U$ we have function $r(f)$  on $\mathcal{F}(U)$ valued in $[0,\infty]$ such that if $r(f_n-f)$ goes to $0$
then $f_n \rightarrow f$ uniformly and $r(f)\rightarrow r(f_n).$ Moreover, if $U$ is bounded and $f$ extends to some open neighborhood of its closure, then $r(f)$ is finite.
\item there is a function $c(x,u,t,U)$ which is monotone in $t$ and finite for finite $t$ such that writing $f(u+xz) = \sum a_n(x,u)z^n$
we have $|a_n(x,u)| \leq c(x,u,r(f),U)^{n+1}.$
\item there is a gentric subsheaf $\mathcal{G}$ such that if $V$ is a bounded open subset of $U$ whose closure is a subset of $U,$  $\mathcal{F}(U)|_V \subseteq \overline{\mathcal{G}(V)}$ where the topology is induced by $r.$
\end{enumerate}
As we have uniform estimates on the radius of convergence we can apply the quantitative germination theorem.
\begin{corollary}\label{nobleclass}
Every function in a noble sheaf is analytic.
\end{corollary}

\subsection{The edge-of-the-wedge theorem}
We see the following ``by definition" version of the edge-of-the-wedge theorem.
\begin{theorem}\label{eow2}
Let $\mathbb{K}$ be an extension of $\mathbb{F}.$
Let $\Gamma \subseteq \mathbb{K}^d$ such that 
\begin{enumerate}
\item $\mathbb{F}^d\subseteq\overline{\Gamma}$
\item for any slice of points of the form $x+zv$ where $x,v \in \mathbb{F}^d$ and $z \in \mathbb{K}$
either intersects $\Gamma$ for all $z \in \mathbb{K}\setminus \mathbb{F}$ or none,
\item there is a spread $X$ of $v$ which admit such intersecting slices.
\end{enumerate}
Let $\mathcal{F}$ be a sheaf such that $\mathcal{F}(U)$ is a set of functions $f$ on $U$ such that
\begin{enumerate}
	\item there is a continuous $g$ on $\Gamma \cup U$ such that $g|_U=f,$
	\item  the function $f(u+tv)$ is analytic, and given $\hat{U}$ a neighborhood of $U$ there is 
	$c(u,v,t,U)$ such that   $|a_n(v,u)| \leq c(x,u,\|g\|_{\hat{U}},U)^{n+1}.$
	\item there is a gentric subsheaf $\mathcal{G}$ such that if $V$ is a bounded open subset of $U$ whose closure is a subset of $U,$  $\mathcal{F}(U)|_V \subseteq \overline{\mathcal{G}(V)}$ where the topology is induced by $r.$
\end{enumerate}
Then, $\mathcal{F}$ is noble.
\end{theorem}
Note that the sheaf of continuous functions analytic on $U \subseteq \mathbb{R}^d$ with a extension to $\Pi^d\cup(-\Pi)^d$ which is analytic there satisfy such assumptions.
(The second to last condition corresponds to a Cauchy type estimate. The final condition relies on the fact that the space of functions is convex and locally bounded and thus closed under convolution with a smooth approximate identity.)
\bibliography{references}
\bibliographystyle{plain}

\end{document}